\documentclass[12pt]{amsart}
\usepackage{mathrsfs}

\usepackage{amssymb,amsfonts,amsthm,amsmath}
\usepackage{epsfig}
\usepackage[utf8]{inputenc}
\usepackage{mathrsfs}
\usepackage{lscape}
\usepackage{amssymb,amsmath,graphicx,color,textcomp, amsthm,bbm,bbold, enumerate,booktabs}
\usepackage{chngcntr}
\usepackage{apptools}
\AtAppendix{\counterwithin{theorem}{section}}
\definecolor{Red}{cmyk}{0,1,1,0}

\definecolor{verde}{cmyk}{1,0,1,0}

\definecolor{loka}{cmyk}{.5,0,1,.5}
\definecolor{azul}{cmyk}{1,1,0,0}

\numberwithin{equation}{section}

\newcommand{\be}{\begin{equation}}
\newcommand{\ee}{\end{equation}}

\newtheorem{corolario}{Corollary}

\newtheorem{teorema}{Theorem}

\newtheorem{lemma}{Lemma}

\begin{document}
\title{Existence, uniqueness, estimation and continuous dependence of the solutions of a nonlinear integral and an integrodifferential equations of fractional order}
\author{J. Vanterler da C. Sousa$^1$}
\address{$^1$ Department of Applied Mathematics, Institute of Mathematics,
 Statistics and Scientific Computation, University of Campinas --
UNICAMP, rua S\'ergio Buarque de Holanda 651,
13083--859, Campinas SP, Brazil\newline
e-mail: {\itshape \texttt{vanterlermatematico@hotmail.com, capelas@ime.unicamp.br }}}
\author{E. Capelas de Oliveira$^1$}

\begin{abstract} By means of two fractional order integral inequalities we investigate the existence and uniqueness of the solutions of the fractional nonlinear Volterra integral equation and a fractional nonlinear integrodifferential equation in Banach space $C_{\xi}$, using an adequate norm, $||\cdot||_{\xi,\infty}$. We estimate the solutions and investigate their continuous dependence.

\end{abstract}
\maketitle

%%%%%%%%%%%%%%%%%%%%%%%%%%%%%%%%%%%%%%%%%%%%%%%%%%%%%%%%%%%%%%%%%%%%%%%%%%%%%%%%%%%%%%%%%%%%%%%%%%%%%%%%%%%%%%%%%%%%%%%%%%%%%%%%%%%%%%%%%%%%%%%%%%%%%%%%%%%%%
\section{Introduction}

The study of fractional differential equations is, in fact, very interesting and important for engineering, physics, chemistry, biology and medicine, among others, for its ability to model and describe natural phenomena \cite{HER,SAMKO,KSTJ,IP,ZE4,ZE5}. It is also of paramount importance for mathematics, in particular, for the fractional analysis, because it allows to study the existence and uniqueness of a class of local and non-local solutions, impulsive problems in the Banach space, nonlocal mild solutions, boundary value problems and many others especially where a differential and/or integral and/or integrodifferential equation emerges \cite{esto1,gronwall,almeida,cong,d1,i5,d3,i7,d5,d6,d7,d8,id1,d10,id6,d12,d13}.

On the other hand, using iterative methods to approximate solutions of fractional integral equations and other numerical studies are, in fact, important and interesting for this area \cite{i1,i2,i3,i4}. In addition, it has been investigated the existence of solutions to a fractional order integral equation by Schauder's fixed point and by singular nonlinear Volterra integral equation \cite{i5,i6}. We also mention the importance of studying the existence, as well as the attractiveness of solutions of fractional order integral equations, in Fréchet spaces, besides the asymptotic behavior of solutions \cite{i7,i8}. It can be said that growth in studying existence and uniqueness of solutions of problems, involving fractional integral equations, has become a field of fractional calculus well seen since it allowed to unify several areas, in particular, an area of mathematics that is growing, is the study of solutions of integral equations via fuzzy \cite{i9,i10,i11}. In this sense, some researchers have decided to study solutions of fuzzy fractional integral equations and have been important for the advancement in this area \cite{i9,i10,i11}.

To the best of our knowledge, there are fractional integrodifferential equations not yet sufficiently investigated. However, some authors have investigated the study of existence and uniqueness for boundary and impulse problems in the Banach space, as well as inverse problems in Sobolev's space \cite{id1,id6,id2,id3,id4,id5,id7,id8,id9,id10,id11}. Besides that, we highlight the study of existence and uniqueness of mild solutions in Sobolev space and 
impulse equations with boundary value \cite{id12}. The variety of problems investigated can be it still does not seem sufficient to cover the vast number of papers related to this subject.

In this paper, we consider nonlinear fractional Volterra integral and integrodifferential equations 
\begin{equation}
x(t)=f\left( t,x(t),\dfrac{1}{\Gamma (\alpha )}\int_{a}^{t}\mathcal{W}_{\psi }^{\alpha }\left( t,s,x\left( s\right) \right) \,ds\right)   \label{1.1}
\end{equation}%
and

\begin{equation}\label{1.2}
\left\{ 
\begin{array}{ccl}
^{H}\mathbb{D}_{a^{+}}^{\alpha ,\beta ,\psi }x(t) & = & f\left( t,x(t),\displaystyle\frac{1}{\Gamma (\alpha )}\int_{a}^{t}\mathcal{W}_{\psi }^{\alpha }\left( t,s,x\left( s\right) \right)ds\right)  \\ 
I_{a^{+}}^{1-\gamma ,\psi }x(a) & = & x_{0}
\end{array}%
\right. 
\end{equation}%
receptively, where ${}^{H}\mathbb{D}_{a^{+}}^{\alpha ,\beta ,\psi }(\cdot )$ is the $\psi $-Hilfer fractional derivative, $I_{a^{+}}^{1-\gamma ,\psi }x(\cdot )$ is the $\psi $-Riemann-Liouville fractional integral, with $
0<\alpha \leq 1$, $0\leq \beta \leq 1$, $\gamma =\alpha +\beta (1-\alpha )$, for $-\infty <a\leq t<+\infty $, being $x,f,k$ real vectors with $n$ components such that $k\in C(I^{2}\times \mathbb{R}^{n},\mathbb{R}^{n})$ for 
$a\leq s\leq t<+\infty $, $f\in C(I\times \mathbb{R}^{n}\times \mathbb{R}^{n},\mathbb{R}^{n})$, and to simplify notation $\mathcal{W}_{\psi }^{\alpha }\left( t,s,x\left( s\right) \right) :=N_{\psi }^{\alpha }\left( t,s\right) k\left(
t,s,x\left( s\right) \right) $ with $N_{\psi }^{\alpha }\left( t,s\right)=\psi ^{\prime }\left( s\right) \left( \psi \left( t\right) -\psi \left(s\right) \right) ^{\alpha -1}$ and $\psi ^{\prime }\left( s\right) =\dfrac{d}{
ds}\psi \left( s\right) $ denoting ordinary derivative.

The main aim of this article is to present an analytical study, that is, existence, uniqueness, solution estimate and continuous dependence of 
solutions of the nonlinear fractional integral equation Eq.(\ref{1.1}) and the nonlinear fractional integrodifferential equation Eq.(\ref{1.2}) in the field of the $\psi$-Hilfer fractional derivative in the Banach space by means of two suitable lemmas.

The paper is organized as follows: in section 2, we present the definitions of $\psi$-Riemann-Liouville fractional integral and $\psi$-Hilfer fractional 
derivative, as well as results relating both and the calculation of the fractional integral of a Mittag-Leffler function. We introduce a norm and the f
undamental metric for the elaboration of this article and two results involving the metric (complete space) and the norm (Banach space), as well as 
discussing particular cases. We also present two lemmas and a corollary involving inequalities that are important for the development of the work. 
In section 3, we investigate the existence and uniqueness of solutions of the nonlinear fractional Volterra integral equation and of the nonlinear fractional 
integrodifferential equation, as well as the study of the solution estimate. Section 4 is intended for the study of the 
continuous dependence of the solutions of the nonlinear fractional Volterra integral equations and of the nonlinear fractional integrodifferential equation. Concluding remarks close the paper.

%%%%%%%%%%%%%%%%%%%%%%%%%%%%%%%%%%%%%%%%%%%%%%%%%%%%%%%%%%%%%%%%%%%%%%%%%%%%%%%%%%%%%%%%%%%%%%%%%%%%%%%%%%%%%%%%%%%%%%%%%%%%%%%%%%%%%%%%%%%%%%%%%%%%%%%%%%%%%%%%%%%%%%%%%%%%%%%%%%%%%%%%

\section{Preliminaries}

In this section we will introduce some preliminary results that will be useful in next sections. Specifically, we will recover some results involving the Riemann-Liouville fractional integral of a function with respect to another function; the $\psi$-Hilfer fractional derivative and the definition of the classical (one parameter) Mittag-Leffler function.

Let $\mathbb{R}^{+} = [0,+\infty)$ be the set of all non-negative real numbers and $\mathbb{R}^n$ be the $n$-dimensional Euclidean space endowed with a norm $\left\Vert \cdot\right\Vert$ and $C(S_1,S_2)$ denotes the class of continuous functions from the set $S_1$ to the set $S_2$.

For any $[a,b] \subset [a,\infty)=I$, let $C([a,b],\mathbb{R}^n)$ be the space of continuous functions $x:[a,b]\to \mathbb{R}^n$ with the sup norm $\left\Vert \cdot\right\Vert_{\infty}$ given by \cite{ZE1}
\begin{equation*}
\left\Vert x\right\Vert _{\infty }:={\smallskip }\underset{{t\in \lbrack a,b] }}{{\sup }}\left\Vert x(t)\right\Vert ,\quad \forall x\in C([a,b],\mathbb{R}^{n}).
\end{equation*}

Let $\alpha >0$, $[a,b]\in \mathbb{R}$ and $\psi (t)$ be an increasing and positive monotone function on $(a,b]$, having a continuous derivative $\psi ^{\prime }(t)$ on $[a,b]$. The Riemann-Liouville fractional integral of a function $f$ with respect to another function $\psi $ on $[a,b]$ is defined by \cite{ZE1,ZE3}
\begin{equation}
I_{a^{+}}^{\alpha ,\psi }x(t)=\frac{1}{\Gamma (\alpha )}\int_{a}^{t}\psi ^{\prime }\left( s\right) \left( \psi \left( t\right) -\psi \left( s\right) \right) ^{\alpha -1}x(s)\,ds
\end{equation}%
where $\Gamma (\cdot )$ is the gamma function, with $0<\alpha \leq 1$.

On the other hand, let $n-1<\alpha <n$ with $n\in \mathbb{N}$, $J=\left[ a,b\right] $ be an interval such that $-\infty \leq a<b\leq +\infty $ and let $f,\psi \in C^{n}([a,b],\mathbb{R})$ be two functions such that $\psi $ is
increasing and $\psi ^{\prime }(t)\neq 0$, for all $t\in J$. The $\psi $-Hilfer fractional derivative denoted by ${}^{H}\mathbb{D}_{a^{+}}^{\alpha,\beta ,\psi }(\cdot )$ of a function $f$ of order $\alpha $ and type $\beta 
$ $\left( 0\leq \beta \leq 1\right) $, is defined by \cite{ZE1,ZE3}
\begin{equation}
{}^{H}\mathbb{D}_{a^{+}}^{\alpha ,\beta ,\psi }x(t)=I_{a^{+}}^{\beta (n-\alpha ),\psi }\left( \frac{1}{\psi ^{\prime }(t)}\frac{d}{dt}\right) ^{n}I_{a^{+}}^{(1-\beta )(n-\alpha ),\psi }x(t).
\end{equation}

The $\psi $-Hilfer fractional derivative of an $n$-dimensional vector function denoted by $x(t)=(x_{1}(t),\cdots ,x_{n}(t))^{T}$, with the superscript $T$ indicates transposition, is defined component wise as 
\begin{equation*}
{}^{H}\mathbb{D}_{a^{+}}^{\alpha ,\beta ;\psi }x(t):=\left( {}^{H}\mathbb{D}_{a^{+}}^{\alpha ,\beta ;\psi }x_{1}(t),\cdots ,{}^{H}\mathbb{D}_{a^{+}}^{\alpha ,\beta ;\psi }x_{n}(t)\right) ^{T}.
\end{equation*}

\begin{teorema}\label{t1} Let $f\in C^{1}\left( J\right) ;$ $0<\alpha \leq 1$ and $0\leq
\beta \leq 1$, we have 
\begin{equation*}
^{H}\mathbb{D}_{a^{+}}^{\alpha ,\beta ;\psi }I_{a^{+}}^{\alpha ,\psi }x\left( t\right) =x\left( t\right) .
\end{equation*}
\end{teorema}

\begin{proof}
See \cite{ZE1}.
\end{proof}

\begin{teorema}\label{t2} If  $f\in C^{n}\left( J\right) ,$  $0<\alpha \leq 1$ and $0\leq \beta \leq 1,$ then
\begin{equation*}
I_{a^{+}}^{\alpha ,\psi }\text{ }^{H}\mathbb{D}_{a^{+}}^{\alpha ,\beta ;\psi }x\left( t\right) =x\left( t\right) -\frac{\left( \psi \left( x\right) -\psi \left( a\right) \right) ^{\gamma -1}}{\Gamma \left( \gamma \right) }%
I^{\left( 1-\beta \right) \left( 1-\alpha \right) ;\psi }x\left( a\right).
\end{equation*}
\end{teorema}

\begin{proof}
See \cite{ZE1}.
\end{proof}

Let $\xi >0$ be a constant and consider the special space $C_{\xi }(I,\mathbb{R}^{n})$ the set of all continuous functions $x\in C(I,\mathbb{R}^{n})$ such that 
\begin{equation*}
\underset{t\in I}{\sup }\frac{\left\Vert x(t)\right\Vert }{\mathbb{E}_{\alpha }\left[ \xi (\psi (t)-\psi (a))^{\alpha }\right] }<\infty 
\end{equation*}%
where $\mathbb{E}_{\alpha }:\mathbb{R}\rightarrow \mathbb{R}$ is the one-parameter Mittag-Leffler function which is given by 
\begin{equation*}
\mathbb{E}_{\alpha }(z)=\sum_{k=0}^{\infty }\frac{z^{k}}{\Gamma (\alpha k+1)},\qquad \forall z\in \mathbb{R}.
\end{equation*}

We couple the linear space $C_{\xi }(I,\mathbb{R}^{n})$ with a suitable metric namely 
\begin{equation}\label{AA1}
{d}_{\xi ,\infty }(x,y):=\underset{t\in I}{\sup }\frac{\left\Vert x(t)-y(t)\right\Vert }{\mathbb{E}_{\alpha }\left[ \xi (\psi (t)-\psi (a))^{\alpha }\right] }<\infty   
\end{equation}
with a norm defined by 
\begin{equation}\label{AA2}
\left\Vert x\right\Vert _{\xi ,\infty }=\underset{t\in I}{\sup }\frac{\left\Vert x(t)\right\Vert }{\mathbb{E}_{\alpha }\left[ \xi (\psi (t)-\psi(a))^{\alpha }\right] }.  
\end{equation}

Note that the metric and norm as seen in Eq.(\ref{AA1}) and Eq.(\ref{AA2}) are in fact an extension of a class of metrics and norms, that is, taking $\psi \left(t\right)=t$ in Eq.(\ref{AA1}) and Eq.(\ref{AA2}), we have
\cite{cong}:
\begin{equation}
{d}_{\xi ,\infty }(x,y):=\underset{t\in I}{\sup }\frac{\left\Vert x(t)-y\left( t\right) \right\Vert }{\mathbb{E}_{\alpha }[\xi \left( t-a\right) ^{\alpha }]}<\infty 
\end{equation}%
with a norm defined by 
\begin{equation}
\left\Vert x\right\Vert _{\xi ,\infty }=\underset{t\in I}{\sup }\frac{\left\Vert x(t)\right\Vert }{\mathbb{E}_{\alpha }[\xi \left( t-a\right)^{\alpha }]}.
\end{equation}

On the other hand, taking $\psi \left( t\right) =t$  and applying limit $\alpha \rightarrow 1$ in the Eq.(\ref{AA1}) and Eq.(\ref{AA2}), we have \cite{pach}
\begin{equation}
{d}_{\xi ,\infty }(x,y):=\underset{t\in I}{\sup }\frac{\left\Vert x(t)-y\left( t\right) \right\Vert }{\exp [\xi \left( t-a\right) ]}<\infty 
\end{equation}
with a norm defined by 
\begin{equation}
\left\Vert x\right\Vert _{\xi ,\infty }=\underset{t\in I}{\sup }\frac{\left\Vert x(t)\right\Vert }{\exp [\xi \left( t-a\right) ]}.
\end{equation}

The definitions given above Eq.(\ref{AA1}) and Eq.(\ref{AA2}) are variants of the metric and norm. When we get the particular case, we note that the respective metrics and norms are variations of norms and metrics of Bielecki \cite{pach}. It may be noted that it is possible to obtain other variances, since the freedom of choice of $\psi$ functions and the limits of $\alpha$ and $\beta$, allows a great advantage in the best metric and norm, in which want to work.

\begin{lemma} {\rm \cite{almeida}} Given $\xi >0$, $n-1<\alpha <n$ with $n\in \mathbb{N}$. Consider the real function $f(t)=\mathbb{E}_{\alpha }\left[ \xi (\psi (t)-\psi (a))^{\alpha }\right] $ where $\mathbb{E}_{\alpha }(\cdot )$ is an one-parameter Mittag-Leffler function. Then 
\begin{equation}
I_{a^{+}}^{\alpha ,\psi }f(t)=\frac{1}{\xi }\left( \mathbb{E}_{\alpha }\left[\xi (\psi (t)-\psi (a))^{\alpha }\right] -1\right) .
\end{equation}
\end{lemma}

\begin{lemma} {\rm \cite{cong}} If $\xi >0$ is a constant, then
\begin{enumerate}
\item ${d}_{\xi ,\infty }$ is a metric;
\item $\left( C_{\xi }\left( I,\mathbb{R}^{n}\right) ,d_{\xi,\infty }\right) $ is a complete metric space.
\end{enumerate}
\end{lemma}

\begin{lemma} {\rm \cite{cong}} If $\xi >0$ is a constant, then
\begin{enumerate}
\item $\left\Vert  \cdot  \right\Vert _{\xi ,\infty }$ is a norm;
\item $\left( C_{\xi }\left( I,\mathbb{R}^{n}\right) ,\left\Vert  \cdot \right\Vert _{\xi ,\infty }\right) $ is a Banach space.
\end{enumerate}
\end{lemma}

We introduce the notation to facilitate of the development of the paper 
\begin{enumerate}
\item $\mathcal{W}_{\psi }^{\alpha }\left( t,s,x\left( s\right) \right) :=N_{\psi }^{\alpha }\left( t,s\right) k\left( t,s,x\left( s\right) \right) $ with $N_{\psi }^{\alpha }\left( t,s\right) =\psi ^{\prime }\left( s\right) \left( \psi \left( t\right) -\psi \left( s\right) \right) ^{\alpha -1}$ and $\psi^{\prime }\left( s\right) =\dfrac{d}{ds}\psi \left( s\right)$;

\item $\mathcal{W}_{\psi }^{\alpha }\left( t,s,0\right) :=N_{\psi }^{\alpha }\left(t,s\right) k\left( t,s,0\right) $;

\item $\overline{\mathcal{W}}_{\psi }^{\alpha }\left( t,s,x\left( s\right) \right) :=N_{\psi }^{\alpha }\left( t,s\right) \overline{k}\left( t,s,x\left( s\right) \right)$;

\item $\Psi ^{\gamma }\left( t,a\right) :=\dfrac{\left( \psi \left( t\right) -\psi \left( a\right) \right) ^{\gamma -1}}{\Gamma \left( \gamma \right) }$.
\end{enumerate}

The proof of Lemma \ref{lemmA} and Corollary \ref{coro1} below, will be omitted here, however it follows the same steps as in Gronwall inequality (Theorem 3) and Corollary 3.10 \cite{gronwall}.

\begin{lemma}\label{lemmA} Let $u(t),v\left( t\right) ,g\left( t\right) \in C\left( I,\mathbb{R}_{+}\right) $, $r(t,\sigma )\in C\left( D,\mathbb{R}_{+}\right) $, where $D=\{(t,\tau )\in I^{2};a\leq \tau \leq +\infty \}$ and $c\geq 0$ is a
constant and $u,v$ are nonnegative and $g$ nonnegative and nondecreasing. If 
\begin{equation}\label{eq12}
u(t)\leq v(t)+g(t)\int_{a}^{t}N_{\psi }^{\alpha }\left( t,\tau \right) r(t,\tau )\,u(\tau )\,d\tau 
\end{equation}%
then 
\begin{equation*}
u(t)\leq v(t)+\int_{a}^{t}\sum_{k=1}^{\infty }\dfrac{(\left( g(t)\Gamma (\alpha )\right) ^{k}}{\Gamma \left( \alpha k\right) }N_{\psi }^{\alpha k}\left( t,\tau \right) r(t,\tau )v(\tau ),d\tau .
\end{equation*}
\end{lemma}

\begin{corolario} \label{coro1} Under the hypothesis of {\rm Lemma \ref{lemmA}}, let $r,v$ be two nondecreasing functions on $I$. Then, we have 
\begin{equation}\label{eq13}
u(t)\leq v(t)\mathbb{E}_{\alpha }[g(t)r(t,t)\Gamma (\alpha )(\psi (t)-\psi
(a))^{\alpha }]
\end{equation}
where $\mathbb{E}_{\alpha }(\cdot )$ is an one-parameter Mittag-Leffler function.
\end{corolario}

It's important to note that, Lemma \ref{lemmA}, is a generalization of the Gronwall inequality, when we take $ r\left(t,\tau \right)=1$ \cite{gronwall}.

\begin{lemma}\label{lemmaB} Let $u\left( t\right) ,p\left( t\right) ,\widetilde{g}\left(t\right) \in C\left( I,\mathbb{R} _{+}\right) ,$ $r\left( t,\sigma \right) 
\in C\left( D,\mathbb{R}_{+}\right) ,$ where $D$ is as in {\rm Lemma \ref{lemmA}} and $\widetilde{g}\left(t\right) \geq 0$ and $u\left( t\right) ,p\left( t\right) $ are nonnegative and $\widetilde{g}\left( t\right) $ nonnegative and nondecreasing. If
\begin{equation}
u(t)\leq \widetilde{g}(t)+\int_{a}^{t}N_{\psi }^{\alpha }\left( t,\tau \right) p(\tau )\left[ u\left( \tau \right) +\int_{a}^{\tau }N_{\psi }^{\alpha }\left( t,\tau \right) r\left( \tau ,\sigma \right) u\left( \sigma
\right) d\sigma \right]d\tau 
\end{equation}
for $t\in I$, then
\begin{equation}
u(t)\leq \widetilde{g}(t)\mathbb{E}_{\alpha }[p(t)\Gamma (\alpha )\mathbb{E}_{\alpha }\left( r(t,t)\Gamma (\alpha )(\psi (t)-\psi (a))^{\alpha }\right)(\psi (t)-\psi (a))^{\alpha }],
\end{equation}
where $\mathbb{E}_{\alpha }(\cdot )$ is an one-parameter Mittag-Leffler function.
\end{lemma}

\begin{proof}
Taking $v\left( t\right) =p\left( t\right) u\left( t\right) $ $\left( \mbox{with}\text{ } p\left( t\right) \text{ and }u\left( t\right) \text{are nonnegative}\right) $ and $g\left( t\right) =p\left( t\right) $ nondecreasing, substituting in {\rm Eq.(\ref{eq12})}, we have
\begin{equation}\label{eq16}
u\left( t\right) \leq p\left( t\right) u\left( t\right) +p\left( t\right) \int_{a}^{t}N_{\psi }^{\alpha }\left( t,\tau \right) r\left( t,\tau \right) u\left( \tau \right) d\tau.
\end{equation}

Applying the integral $\displaystyle \int_{a}^{t}\psi ^{\prime }\left( \tau \right) \left(\psi \left( t\right) -\psi \left( \tau \right) \right) ^{\alpha -1}d\tau $ and summing $\widetilde{g}\left( t\right) $ on both sides of Eq.(\ref{eq16}), we get
\begin{eqnarray*}
&&\widetilde{g}\left( t\right) +\int_{a}^{t}N_{\psi }^{\alpha }\left( t,\tau
\right) u\left( \tau \right) d\tau  \\
&\leq &\widetilde{g}\left( t\right) +\int_{a}^{t}N_{\psi }^{\alpha }\left(
t,\tau \right) \left[ p\left( \tau \right) u\left( \tau \right) +p\left(
\tau \right) \int_{a}^{\tau }N_{\psi }^{\alpha }\left( s,\sigma \right)
r\left( \tau ,\sigma \right) u\left( \sigma \right) d\sigma \right] d\tau .
\end{eqnarray*}

Therefore, we conclude that 
\begin{equation}\label{eq17}
u\left( t\right) \leq \widetilde{g}\left( t\right) +\int_{a}^{t}N_{\psi}^{\alpha }\left( t,\tau \right) \left[ p\left( \tau \right) u\left( \tau\right) +p\left( \tau \right) \int_{a}^{\tau }N_{\psi }^{\alpha }\left(
s,\sigma \right) r\left( \tau ,\sigma \right) u\left( \sigma \right) d\sigma \right] d\tau.
\end{equation}

Note that, the Eq.(\ref{eq17}) is exactly the hypotheses of this Lemma.

On the other hand, we will perform the same procedure as in Lemma \ref{lemmA}, i.e., Eq.(\ref{eq12}). Then, taking $v\left( t\right) =p\left( t\right) u\left( t\right) $ $\left(\text{with }p\left( t\right) \text{ and }u\left( t\right) \text{ nonnegative}\right) $ and $g\left( t\right) =p\left( t\right) $ nondecreasing in Eq.(\ref{eq12}), we have
\begin{equation}\label{eq18}
u\left( t\right) \leq p\left( t\right) u\left( t\right) \mathbb{E}_{\alpha }\left[ p\left( t\right) r\left(t,t\right) \Gamma \left( \alpha \right)\left( \psi \left( t\right) -\psi \left( a\right) \right) ^{\alpha }\right] .
\end{equation}

Applying the integral $\displaystyle\int_{a}^{t}\psi ^{\prime }\left( \tau \right) \left( \psi \left( t\right) -\psi \left( \tau \right) \right) ^{\alpha -1}d\tau $ and summing $\widetilde{g}\left( t\right) $ on both sides of Eq.(\ref{eq18}), we get
\begin{equation*}
u(t)\leq \widetilde{g}(t)\left( 1+\int_{a}^{t}N_{\psi }^{\alpha }\left(t,\tau \right) p(\tau )\mathbb{E}_{\alpha }\left[ p\left( t\right) r\left(\tau ,\tau \right) \Gamma \left( \alpha \right) \left( \psi \left( \tau
\right) -\psi \left( a\right) \right) ^{\alpha }\right] u\left( \tau \right)\,d\tau \right) .
\end{equation*}

Using the Lemma \ref{lemmA}, we have
\begin{equation*}
u(t)\leq \widetilde{g}(t)\mathbb{E}_{\alpha }[p(t)\Gamma (\alpha )\mathbb{E}_{\alpha }\left( r(t,t)\Gamma (\alpha )(\psi (t)-\psi (a))^{\alpha }\right)(\psi (t)-\psi (a))^{\alpha }].
\end{equation*}

Thus, we conclude the proof.
\end{proof}

%%%%%%%%%%%%%%%%%%%%%%%%%%%%%%%%%%%%%%%%%%%%%%%%%%%%%%%%%%%%%%%%%%%%%%%%%%%%%%%%%%%%%%%%%%%%%%%%%%%%%%%%%%%%%%%

\section{Existence, uniqueness and estimates on the solutions}

In this section we are going to present our main results concerning the existence and uniqueness of solutions of Eq.(\ref{1.1}) and Eq.(\ref{1.2}).

\begin{teorema}\label{Theo1} Let $L>0$, $\xi >0$, $M\geq 0$, $\delta >1$ be constants with $\xi =L_{\delta }$. Suppose the  functions $f,k$ in {\rm{Eq.(\ref{1.1})}} satisfying the conditions 
\begin{equation}
\left\Vert f(t,u,v)-f(t,\overline{u},\overline{v})\right\Vert \leq M\left( \left\Vert u-\overline{u}\right\Vert +\left\Vert v-\overline{v}\right\Vert\right)   \label{AB}
\end{equation}
and 
\begin{equation}
\left\Vert k(t,s,u)-k(t,s,\overline{u})\right\Vert \leq L\left\Vert u-\overline{u}\right\Vert   \label{BA}
\end{equation}
and 
\begin{equation}
d_{1}=\underset{t\in I}{\sup }\frac{1}{\mathbb{E}_{\alpha }[\xi \left( \psi\left( t\right) -\psi \left( a\right) \right) ^{\alpha }]}\left\Vert f\left(t,0,\frac{1}{\Gamma (\alpha )}\int_{a}^{t}\mathcal{W}_{\psi }^{\alpha }\left(
t,s,0\right) \,ds\right) \right\Vert <\infty 
\end{equation}%
with $\mathcal{W}_{\psi }^{\alpha }\left( t,s,0\right)=N_{\psi }^{\alpha }\left(t,s\right) K\left( t,s,0\right)$.

If $M(1+1/\delta )<1$, then the integral in {\rm{Eq.(\ref{1.1})}} has a unique solution $x\in C_{\xi }\left( I,\mathbb{R}^{n}\right)$.
\end{teorema}

\begin{proof}
First, note that the nonlinear fractional Volterra integral equation given by {\rm{Eq.(\ref{1.1})}} can be rewritten in the following form
\begin{eqnarray}
x(t) &=&f\left( t,x(t),\frac{1}{\Gamma (\alpha )}\int_{a}^{t}\mathcal{W}_{\psi}^{\alpha }\left( t,s,x\left( s\right) \right) \,ds\right) \,- \notag \\
&&f\left( t,0,\frac{1}{\Gamma (\alpha )}\int_{a}^{t}\mathcal{W}_{\psi }^{\alpha}\left( t,s,0\right) \,ds\right) +f\left( t,0,\frac{1}{\Gamma (\alpha )}\int_{a}^{t}\mathcal{W}_{\psi }^{\alpha }\left( t,s,0\right) ds\right)\notag \\   \label{T0}
\end{eqnarray}
for $t\in I$.

Now, for $x\in C_{\xi }\left( I,\mathbb{R}^{n}\right)$, we define the following operator, $T$, by means of 
\begin{eqnarray}
(Tx)(t) &=&f\left( t,x(t),\frac{1}{\Gamma (\alpha )}\int_{a}^{t}\mathcal{W}_{\psi}^{\alpha }\left( t,s,x\left( s\right) \right) \,ds\right) \,- \notag \\
&&f\left( t,0,\frac{1}{\Gamma (\alpha )}\int_{a}^{t}\mathcal{W}_{\psi }^{\alpha}\left( t,s,0\right) ds\right) +f\left( t,0,\frac{1}{\Gamma (\alpha )}\int_{a}^{t}\mathcal{W}_{\psi }^{\alpha }\left( t,s,0\right) \,ds\right).\notag \\  \label{T1}
\end{eqnarray}%
Using \textrm{Eq.(\ref{T1})} and hypotheses, we get 
\begin{equation}
\left\Vert Tx\right\Vert _{\xi ,\infty }=\underset{t\in I}{\sup }\frac{\left\Vert (Tx)(t)\right\Vert }{{\mathbb{E}}_{\alpha }\left[ \xi \left( \psi\left( t\right) -\psi \left( a\right) \right) ^{\alpha }\right] }.
\label{T1.1}
\end{equation}

Thus {\rm{Eq.(\ref{T1.1})}} can be written as 
\begin{eqnarray}\label{T1.2}
&&\left\Vert Tx\right\Vert _{\xi ,\infty }  \notag \\
&=&\underset{t\in I}{\sup }\frac{1}{{\mathbb{E}}_{\alpha }\left[ \xi \left( \psi \left( t\right) -\psi \left( a\right) \right) ^{\alpha }\right] }\left\Vert f\left( t,x(t),\frac{1}{\Gamma \left( \alpha \right) }%
\int_{a}^{t}\mathcal{W}_{\psi }^{\alpha }\left( t,s,x\left( s\right) \right) ds\right) \right. \,-  \notag \\
&&\left. f\left( t,0,\frac{1}{\Gamma \left( \alpha \right) }\int_{a}^{t}\mathcal{W}_{\psi }^{\alpha }\left( t,s,0\right) \,ds\right) +f\left( t,0,\frac{1}{\Gamma \left( \alpha \right) }\int_{a}^{t}\mathcal{W}_{\psi }^{\alpha }\left(
t,s,0\right) ds\right) \right\Vert   \notag \\
&\leq &\underset{t\in I}{\sup }\frac{1}{{\mathbb{E}}_{\alpha }\left[ \xi \left( \psi \left( t\right) -\psi \left( a\right) \right) ^{\alpha }\right] }\left\Vert f\left( t,x(t),\frac{1}{\Gamma \left( \alpha \right) }%
\int_{a}^{t}\mathcal{W}_{\psi }^{\alpha }\left( t,s,x\left( s\right) \right)\,ds\right) \right. -  \notag \\
&&\left. f\left( t,0,\frac{1}{\Gamma \left( \alpha \right) }\int_{a}^{t}\mathcal{W}_{\psi }^{\alpha }\left( t,s,0\right) \,ds\right) \right\Vert + \notag \\
&&\underset{t\in I}{\sup }\frac{1}{{\mathbb{E}}_{\alpha }\left[ \xi \left( \psi \left( t\right) -\psi \left( a\right) \right) ^{\alpha }\right] } \left\Vert f\left( t,0,\frac{1}{\Gamma (\alpha }\int_{a}^{t}\mathcal{W}_{\psi
}^{\alpha }\left( t,s,0\right) \,ds\right) \right\Vert   \notag \\
&=&d_{1}+\underset{t\in I}{\sup }\frac{1}{{\mathbb{E}}_{\alpha }\left[ \xi \left( \psi \left( t\right) -\psi \left( a\right) \right) ^{\alpha }\right] }\left\Vert f\left( t,x(t),\frac{1}{\Gamma \left( \alpha \right) }%
\int_{a}^{t}\mathcal{W}_{\psi }^{\alpha }\left( t,s,x\left( s\right) \right) \,ds\right) \right. -  \notag \\
&&\left. f\left( t,0,\frac{1}{\Gamma \left( \alpha \right) } \int_{a}^{t}\mathcal{W}_{\psi }^{\alpha }\left( t,s,0\right) \,ds\right) \right\Vert  \notag \\
\end{eqnarray}
\begin{eqnarray}\label{T1.2}
&\leq &d_{1}+\underset{t\in I}{\sup }\frac{M}{{\mathbb{E}}_{\alpha }\left[ \xi \left( \psi \left( t\right) -\psi \left( a\right) \right) ^{\alpha } \right] }\left\{ \left\Vert x(t)\right\Vert +\left\Vert \frac{1}{\Gamma
(\alpha )}\int_{a}^{t}\mathcal{W}_{\psi }^{\alpha }\left( t,s,x\left( s\right) \right) \,ds\right. \right. -  \notag \\
&&\left. \left. \frac{1}{\Gamma (\alpha )}\int_{a}^{t}\mathcal{W}_{\psi }^{\alpha }\left( t,s,0\right) \,ds\right\Vert \right\}   \notag \\ &\leq &d_{1}+\underset{t\in I}{\sup }\frac{M}{{\mathbb{E}}_{\alpha }\left[
\xi \left( \psi \left( t\right) -\psi \left( a\right) \right) ^{\alpha } \right] }\left\{ \left\Vert x(t)\right\Vert +\frac{1}{\Gamma (\alpha )} \int_{a}^{t}N_{\psi }^{\alpha }(t,s)\,L\left\Vert x\left( s\right)
\right\Vert \,ds\right\}.\notag \\
\end{eqnarray}

Manipulating {\rm{Eq.(\ref{T1.2})}} we can write 
\begin{eqnarray}
&&\left( Tx\right) (t)\notag \\ &\leq &d_{1}+M\underset{t\in I}{\sup }\frac{\left\Vert x\left( t\right) \right\Vert }{{\mathbb{E}}_{\alpha }\left[ \xi \left( \psi \left( t\right) -\psi \left( a\right) \right) ^{\alpha }\right] }+ML\underset {t\in I}{\sup }\frac{\left\Vert x\left( t\right) \right\Vert }{{\mathbb{E}}_{\alpha }\left[ \xi \left( \psi \left( t\right) -\psi \left( a\right) \right) ^{\alpha }\right] } \notag \\
&&\underset{t\in I}{\sup }\frac{1}{{\mathbb{E}}_{\alpha }\left[ \xi \left( \psi \left( t\right) -\psi \left( a\right) \right) ^{\alpha }\right] }\left\{ \frac{1}{\Gamma (\alpha )}\int_{a}^{t}N_{\psi }^{\alpha }(t,s){%
\mathbb{E}}_{\alpha }\left[ \xi \left( \psi \left( t\right) -\psi \left(a\right) \right) ^{\alpha }\right] \,ds\right\} \notag\\
&=&d_{1}+M+ML\left\Vert x\right\Vert _{\xi ,\infty }\underset{t\in I}{\sup }\frac{1}{{\mathbb{E}}_{\alpha }\left[ \xi \left( \psi \left( t\right) -\psi\left( a\right) \right) ^{\alpha }\right] }\left\{ \frac{1}{\xi }\left( {%
\mathbb{E}}_{\alpha }\left[ \xi \left( \psi \left( t\right) -\psi \left(a\right) \right) ^{\alpha }\right] -1\right) \right\}  \notag \\
&=&d_{1}+M\left\Vert x\right\Vert _{\xi ,\infty }\left[ 1+\frac{L}{\xi }\underset{t\in I}{\sup }\left( 1-\frac{1}{{\mathbb{E}}_{\alpha }\left[ \xi\left( \psi \left( t\right) -\psi \left( a\right) \right) ^{\alpha }\right] }%
\right) \right] .
\end{eqnarray}%

Since $\mathbb{E}_{\alpha }(\cdot )$ is a monotone increasing function on real line, we have 
\begin{equation}
\left\Vert Tx\right\Vert _{\xi ,\infty }\leq d_{1}+M\left\Vert x\right\Vert_{\xi ,\infty }\left( 1+\frac{L}{\xi }\right) =d_{1}+\left\Vert x\right\Vert_{\xi ,\infty }M\left( 1+\frac{1}{\delta }\right) <\infty .
\end{equation}

Therefore, the operator $T$ maps $C_{\xi }(I,\mathbb{R}^{n})$ into itself, i.e., 
\begin{equation}
T\left( \left( C_{\xi }\left( I,\mathbb{R}^{n}\right) ,\left\Vert \left( \cdot \right) \right\Vert _{\xi ,\infty }\right) \right) \subset \left( C_{\xi }\left( I,\mathbb{R}^{n}\right) ,\left\Vert \left( \cdot \right)
\right\Vert _{\xi ,\infty }\right).
\end{equation}

We now show the operator $T$ is a contraction. Let $u,v\in C_{\xi }\left( I,\mathbb{R}^{n}\right) $, then, by {\rm{Eq.(\ref{T1})}} and hypotheses, we get 
\begin{eqnarray}\label{T1.3}
&&d_{\xi ,\infty }\left( Tu,Tv\right)\notag \\  &=&\underset{t\in I}{\sup }\frac{\left\Vert (Tu)(t)-(Tv)(t)\right\Vert }{{\mathbb{E}}_{\alpha }\left[ \xi\left( \psi \left( t\right) -\psi \left( a\right) \right) ^{\alpha }\right] } \notag \\ &=&\underset{t\in I}{\sup }\frac{
\begin{array}{l}
\left\Vert f\left( t,u(t),\displaystyle\frac{1}{\Gamma (\alpha )}\int_{a}^{t}\mathcal{W}_{\psi}^{\alpha }\left( t,s,u\left( s\right) \right) \,ds\right) \right. - \\ \left. f\left( t,v(t),\displaystyle\frac{1}{\Gamma (\alpha )}\int_{a}^{t}\mathcal{W}_{\psi}^{\alpha }\left( t,s,v\left( s\right) \right) \,ds\right) \right\Vert 
\end{array}
}{{\mathbb{E}}_{\alpha }\left[ \xi \left( \psi \left( t\right) -\psi \left(a\right) \right) ^{\alpha }\right] }.
\end{eqnarray}

As above, manipulating {\rm{Eq.(\ref{T1.3})}}, we can write 
\begin{equation}
d_{\xi ,\infty }\left( Tu,Tv\right) =M d_{\xi ,\infty }\left[ 1+\frac{L}{\xi }\underset{t\in I}{\sup }\left( 1-\frac{1}{{\mathbb{E}}_{\alpha }\left[ \xi\left( \psi \left( t\right) -\psi \left( a\right) \right) ^{\alpha }\right] }\right) \right] .
\end{equation}%
Since $\mathbb{E}_{\alpha }(\cdot )$ is a monotone increasing function on real line, we have 
\begin{equation}
d_{\xi ,\infty }\left( Tu,Tv\right) \leq Md_{\xi ,\infty }\left( u,v\right) \left( 1+\frac{L}{\xi }\right) =M\left( 1+\frac{1}{\delta }\right) d_{\xi,\infty }(u,v).
\end{equation}

By hypotheses, $M(1+1/\delta )<1$, then by Banach fixed point theorem \cite{ZE4,ZE5}, the operator, $T$, has a unique fixed point in $C_{\xi }\left( I,\mathbb{R}^{n}\right) $. Thus, we conclude that, the fixed point of $T$ is however a solution of {\rm{Eq.(\ref{1.1})}}. 
\end{proof}

\begin{teorema}
Let $L$, $\xi $, $M$, $\delta $ be as in {\rm {Theorem (\ref{Theo1})}}. Suppose the functions $f$ and $k$ in {\rm{Eq.(\ref{1.2})}} satisfying the conditions given in {\rm{Eq.(\ref{AB})}} and {\rm{Eq.(\ref{BA})}} and the relation 
\begin{equation*}
d_{2}=\underset{t\in I}{\sup }\frac{1}{{\mathbb{E}}_{\alpha }\left[ \xi \left( \psi \left( t\right) -\psi \left( a\right) \right) ^{\alpha }\right] } \left\Vert \Psi ^{\gamma }\left( t,a\right) x_{0}-I_{a^{+}}^{\alpha ,\psi }f\left( s,0,\frac{1}{\Gamma (\alpha )} \int_{a}^{s}\mathcal{W}_{\psi }^{\alpha }\left( s,\sigma ,0\right) \,d\sigma \right)\right\Vert .
\end{equation*}
If $\displaystyle\frac{M}{\xi }\left( 1+\frac{1}{\delta }\right) <1$, then the nonlinear integrodifferential equation {\rm{Eq.(\ref{1.2})}} has a unique solution $x\in C_{\xi }\left( I,\mathbb{R}^{n}\right) $.
\end{teorema}

\begin{proof}
We only present the idea of the proof, following the same steps as in the above theorem. First, we will show that {\rm{Eq.(\ref{1.2})}} is equivalent to the 
following nonlinear integral equation 
\begin{equation}
x(t)=\Psi ^{\gamma }\left( t,a\right) x_{0}+I_{a^{+}}^{\alpha ,\psi }f\left( s,x(s),\frac{1}{\Gamma (\alpha )}\int_{a}^{s}\mathcal{W}_{\psi }^{\alpha }\left( s,\sigma ,x\left( \sigma \right) \right) \,d\sigma \right) .  \label{AA}
\end{equation}

In fact, applying the fractional derivative ${}^{H}{\mathbb{D}} _{a^{+}}^{\alpha ,\beta ,\psi }(\cdot )$ on both sides of {\rm{ Eq.(\ref{AA})}} and using {\rm Theorem \ref{t1}}, we get 
\begin{eqnarray}\label{BB}
&&{}^{H}{\mathbb{D}}_{a^{+}}^{\alpha ,\beta ,\psi }x(t)\notag \\ &=&\text{ }^{H}{\mathbb{D}}_{a^{+}}^{\alpha ,\beta ,\psi }\left[ \Psi ^{\gamma }\left(t,a\right) x_{0}\right] +{}^{H}{\mathbb{D}}_{a^{+}}^{\alpha ,\beta ,\psi }%
\left[ I_{a^{+}}^{\alpha ,\psi }f\left( s,x(s),\frac{1}{\Gamma (\alpha )}\int_{a}^{s}\mathcal{W}_{\psi }^{\alpha }\left( s,\sigma ,x\left( \sigma \right)\right) \,\,d\sigma \right) \right]   \notag \\
&=&f\left( t,x(t),\frac{1}{\Gamma (\alpha )}\int_{a}^{t}\mathcal{W}_{\psi }^{\alpha }\left( t,s,x\left( s\right) \right) \,ds\,\right) 
\end{eqnarray}%
where 
\begin{equation*}
{}^{H}{\mathbb{D}}_{a^{+}}^{\alpha ,\beta ,\psi }\left[ \Psi ^{\gamma }\left( t,a\right) x_{0}\right] =0.
\end{equation*}

Applying the fractional integral $I_{a^{+}}^{\alpha ,\psi }\left( \cdot \right) $ on both sides of \thinspace {\rm{Eq.(\ref{BB})}} and using {\rm Theorem \ref{t2}}, we get 
\begin{equation*}
x(t)=\Psi ^{\gamma }\left( t,a\right) x_{0}+I_{a^{+}}^{\alpha ,\psi }f\left( s,x(s),\frac{1}{\Gamma (\alpha )}\int_{a}^{s}\mathcal{W}_{\psi }^{\alpha }\left( s,\sigma ,x\left( \sigma \right) \right) \,d\sigma \right) .
\end{equation*}

Let $x\in C_{\xi }\left( I,\mathbb{R}^{n}\right) $ and consider the following operator, $S$, given by 
\begin{eqnarray*}
&&(Sx)(t) \notag \\&=&\Psi ^{\gamma }\left( t,a\right) x_{0}+I_{a^{+}}^{\alpha ,\psi }f\left( s,x(s),\frac{1}{\Gamma (\alpha )}\int_{a}^{s}\mathcal{W}_{\psi }^{\alpha }\left( s,\sigma ,x\left( \sigma \right) \right) \,d\sigma \right) \,- 
\notag \\
&&I_{a^{+}}^{\alpha ,\psi }f\left( s,0,\frac{1}{\Gamma (\alpha )}\int_{a}^{s}\mathcal{W}_{\psi }^{\alpha }\left( s,\sigma ,0\right) d\sigma \right)+I_{a^{+}}^{\alpha ,\psi }f\left( s,0,\frac{1}{\Gamma (\alpha )}\int_{a}^{s}\mathcal{W}_{\psi }^{\alpha }\left( s,\sigma ,0\right) \,d\sigma \right)
\end{eqnarray*}
for $t\in I$.

The proof of the $S\left( (C_{\xi }\left( I,\mathbb{R}^{n}\right),\left\Vert \left( \cdot \right) \right\Vert _{\xi ,\infty }\right) \subset\left( C_{\xi }(\left( I,\mathbb{R}^{n}\right) ,\left\Vert \left( \cdot\right) \right\Vert _{\xi ,\infty }\right) $ and that $S$ is a contraction, is realized with small and appropriated modifications from the proof of the {\rm Theorem \ref{Theo1}}.
\end{proof}

The next theorems, we will investigate the estimation of the solution of the nonlinear fractional Volterra integral equation  and the nonlinear fractional 
integrodifferential equation. Then, we first carried out the estimation investigation for the nonlinear fractional Volterra integral equation.

\begin{teorema}\label{Theo3} Suppose the functions $f,k$ in {\rm{Eq.(\ref{1.1})}} satisfying the conditions 
\begin{equation}
\left\vert f\left( t,u,v\right) -f\left( t,\overline{u},\overline{v}\right) \right\vert \leq N\left( \left\vert u-\overline{u}\right\vert +\left\vert v- \overline{v}\right\vert \right)   \label{AB1}
\end{equation}%
and 
\begin{equation}
|k(t,\sigma ,u)-k(t,\sigma ,v)|\leq r(t,\sigma )|u-v|  \label{BA1}
\end{equation}%
where $0\leq N<1$ is a constant and $r(t,\sigma )\in C(D,\mathbb{R}_{+})$, in which $D=\{(t,\tau )\in I^{2}:a\leq \sigma \leq t<\infty \}$.

Let 
\begin{equation*}
C_{1}=\underset{t\in I}{\sup }\left\vert f\left( t,0,\frac{1}{\Gamma (\alpha )}\int_{a}^{t}\mathcal{W}_{\psi }^{\alpha }\left( t,\sigma ,0\right) \,d\sigma \right) \right\vert <\infty .
\end{equation*}
If $x(t)$, $t\in I$, is any solution of {\rm{Eq.(\ref{1.1})}}, then 
\begin{equation*}
\left\vert x(t)\right\vert \leq \left( \frac{C_{1}}{1-N}\right) \mathbb{E}_{\alpha }\left[ \frac{N}{1-N}r(t,t)(\psi (t)-\psi (a))^{\alpha }\right] 
\end{equation*}
for $t\in I$ and $\mathbb{E}_{\alpha }(\cdot )$ is an one-parameter Mittag-Leffler function.
\end{teorema}

\begin{proof}
Using the fact that the solution $x(t)$ of {\rm{Eq.(\ref{1.1})}} is equivalent to {\rm{Eq.(\ref{T0})}} and the hypotheses, we have 
\begin{eqnarray}\label{3.5}
\left\vert x(t)\right\vert  &\leq &\left\vert f\left( t,0,\frac{1}{\Gamma (\alpha )}\int_{a}^{t}\mathcal{W}_{\psi }^{\alpha }\left( t,\sigma ,0\right) \,d\sigma \right) \right\vert +\left\vert f\left( t,x(t),\frac{1}{\Gamma (\alpha )}\int_{a}^{t}\mathcal{W}_{\psi }^{\alpha }\left( t,\sigma ,x\left( \sigma \right) \right) d\sigma \right) \right\vert   \notag \\ &&-\left\vert f\left( t,0,\frac{1}{\Gamma (\alpha )}\int_{a}^{t}\mathcal{W}_{\psi }^{\alpha }\left( t,\sigma ,0\right) \,d\sigma \right) \right\vert   \notag \\
&\leq &\underset{t\in I}{\sup }\left\vert f\left( t,0,\frac{1}{\Gamma (\alpha )}\int_{a}^{t}\mathcal{W}_{\psi }^{\alpha }\left( t,\sigma ,0\right) \,d\sigma \right) \right\vert + \notag \\
&&N\left\{ |x(t)|+\frac{1}{\Gamma (\alpha )}\int_{a}^{t}N_{\psi }^{\alpha }(t,\sigma )|k(t,\sigma ,x(\sigma )-k(t,\sigma ,0)|\right\}   \notag \\
&\leq &C_{1}+N\text{ }\left\vert x(t)\right\vert +\frac{N}{\Gamma (\alpha )}\int_{a}^{t}N_{\psi }^{\alpha }(t,\sigma )r(t,\sigma )|x(\sigma )|\,d\sigma.
\end{eqnarray}

From {\rm{Eq.(\ref{3.5})}} and the fact that $0\leq N<1$, we get
\begin{equation}
\left\vert x(t)\right\vert \leq \frac{C_{1}}{1-N}+\frac{N}{1-N}\frac{1}{\Gamma (\alpha )}\int_{a}^{t}\psi ^{\prime }\left( \sigma \right) \left(\psi \left( t\right) -\psi \left( \sigma \right) \right) ^{\alpha
-1}r(t,\sigma )\left\vert x(\sigma )\right\vert \,d\sigma .  \label{A1}
\end{equation}

Using Corollary \ref{coro1}, we conclude that 
\begin{equation*}
\left\vert x(t)\right\vert \leq \left( \frac{C_{1}}{1-N}\right) \mathbb{E}_{\alpha }\left[ \frac{N}{1-N}r(t,t)(\psi (t)-\psi (a))^{\alpha }\right] 
\end{equation*}%
where $\mathbb{E}_{\alpha }(\cdot )$ is an one-parameter Mittag-Leffler function.
\end{proof}

In the same way that we present the solution estimate for Eq.(\ref{1.1}), in this sense, we also present an investigation for the estimation of the solution of the integrodifferential equation Eq.(\ref{1.2}).

\begin{teorema}
Suppose the function $f$ in {\rm{Eq.(\ref{1.2})}} satisfying the condition 
\begin{equation}
\left\vert f\left( t,u,v\right) -f\left( t,\overline{u},\overline{v}\right) \right\vert \leq p(t)\left( \left\vert u-\overline{u}\right\vert +\left\vert v-\overline{v}\right\vert \right)   \label{AB3}
\end{equation}%
where $p(t)\in C\left( I,\mathbb{R}_{+}\right) $ and the function $k$ in {\rm{Eq.(\ref{1.2})}} satisfying the condition {\rm{Eq.(\ref{BA1})}}. 

Let 
\begin{equation*}
C_{2}=\underset{t\in I}{\sup }\left\vert \Psi ^{\gamma }\left( t,a\right) x_{0}+I_{a^{+}}^{\alpha ,\psi }f\left( s,0,\frac{1}{\Gamma (\alpha )}\int_{a}^{s}\mathcal{W}_{\psi }^{\alpha }\left( s,\sigma ,0\right) \,d\sigma \right)
\right\vert <\infty .
\end{equation*}

If $x(t)$, $t\in I$, is any solution of {\rm{Eq.(\ref{1.2})}}, then 
\begin{equation*}
|x(t)|\leq C_{2}\mathbb{E}_{\alpha }\left\{ p\left( t\right) \Gamma \left( \alpha \right) \mathbb{E}_{\alpha }\left[ r\left( t,t\right) \Gamma \left( \alpha \right) \left( \psi \left( t\right) -\psi \left( a\right) \right) ^{\alpha }\right]
\left( \psi \left( t\right) -\psi \left( a\right) \right) ^{\alpha }\right\}.
\end{equation*}
\end{teorema}

\begin{proof}
Using the fact that $x(t)$ is a solution of {\rm{Eq.(\ref{1.2})}}, the hypotheses and using the {\rm Lemma \ref{lemmaB}}, we have 
\begin{eqnarray}
&&|x(t)|\notag \\ &=&\left\vert \Psi ^{\gamma }\left( t,a\right)
x_{0}+I_{a^{+}}^{\alpha ,\psi }f\left( s,x(s),\frac{1}{\Gamma (\alpha )}\int_{a}^{s}\mathcal{W}_{\psi }^{\alpha }\left( s,\sigma ,x\left( \sigma \right)\right) \,d\sigma \right) -\right.  \notag \\
&&\left. I_{a^{+}}^{\alpha ,\psi }f\left( s,0,\frac{1}{\Gamma (\alpha )}\int_{a}^{s}\mathcal{W}_{\psi }^{\alpha }\left( s,\sigma ,0\right) \,d\sigma \right)+I_{a^{+}}^{\alpha ,\psi }f\left( s,0,\frac{1}{\Gamma (\alpha )}
\int_{a}^{s}\mathcal{W}_{\psi }^{\alpha }\left( s,\sigma ,0\right) \,d\sigma \right)\right\vert \notag \\
&\leq &\left\vert \Psi ^{\gamma }\left( t,a\right) x_{0}+I_{a^{+}}^{\alpha ,\psi }f\left( s,0,\frac{1}{\Gamma (\alpha )}\int_{a}^{s}\mathcal{W}_{\psi }^{\alpha }\left( s,\sigma ,0\right) \,d\sigma \right) \right\vert +  \notag \\
&&\left\vert I_{a^{+}}^{\alpha ,\psi }f\left( s,x(s),\frac{1}{\Gamma (\alpha )}\int_{a}^{s}\mathcal{W}_{\psi }^{\alpha }\left( s,\sigma ,x\left( \sigma \right)\right) d\sigma \right) -I_{a^{+}}^{\alpha ,\psi }f\left( s,0,\frac{1}{ \Gamma (\alpha )}\int_{a}^{s}\mathcal{W}_{\psi }^{\alpha }\left( s,\sigma ,0\right)\,d\sigma \right) \right\vert   \notag \\ &\leq &C_{2}+I_{a^{+}}^{\alpha ,\psi }\left\{ p(s)\left( |x(s)|+\frac{1}{\Gamma (\alpha )}\int_{a}^{s}N_{\psi }^{\alpha }(s,\sigma )|k(s,\sigma,x(\sigma ))-k(s,\sigma ,0)|\,d\sigma \right) \right\} \notag \\
&\leq &C_{2}+I_{a^{+}}^{\alpha ,\psi }\left\{ p(s)\left( |x(s)|+\frac{1}{\Gamma (\alpha )}\int_{a}^{s}N_{\psi }^{\alpha }(s,\sigma )r(s,\sigma)|x(\sigma )|\,d\sigma \right) \right\}  \notag \\
&\leq &C_{2}\mathbb{E}_{\alpha }\left\{ p\left( t\right) \Gamma \left( \alpha \right)\mathbb{E}_{\alpha }\left[ r\left( t,t\right) \Gamma \left( \alpha \right) \left(\psi \left( t\right) -\psi \left( a\right) \right) ^{\alpha }\right] \left( \psi \left( t\right) -\psi \left( a\right) \right) ^{\alpha }\right\} .
\end{eqnarray}
\end{proof}

%%%%%%%%%%%%%%%%%%%%%%%%%%%%%%%%%%%%%%%%%%%%%%%%%%%%%%%%%%%%%%%%%%%%%%%%%%%%%%%%%%%%%%%%%%%%%%%%%%%%%%%%%%%%%%%%%%%%%%%%%%%%%%%%%%%%%%%%%%%%%%%%%%%%%%%%%%%%%%%%%%%%%%%%%%%%%%%%%%%%%%%%%%%%%%%%%%%%%%%%%%%%%%%%%%%%%%%%%%%%%%%%%%%%%%%%%%%%%%%%%%%

\section{Continuous dependence}

In this section, we present results regarding the continuous dependence of the solutions of Eq.(\ref{1.1}) and Eq.(\ref{1.2}).

Consider Eq.(\ref{1.1}) and Eq.(\ref{1.2}) and the corresponding equations 
\begin{equation}\label{4.1}
y(t)=\overline{f}\left( t,y(t),\frac{1}{\Gamma (\alpha )}\int_{a}^{t}\overline{\mathcal{W}}_{\psi }^{\alpha }\left( t,\sigma ,x\left( \sigma \right)\right) \,\,d\sigma \right)   
\end{equation}
and
\begin{equation}\label{4.2}
\left\{ 
\begin{array}{ccl}
^{H}\mathbb{D}_{a^{+}}^{\alpha ,\beta ;\psi }y(t) & = & \overline{f}\left( t,y(t),\displaystyle\frac{1}{\Gamma (\alpha )}\int_{a}^{t}\overline{\mathcal{W}}_{\psi }^{\alpha }\left( t,\sigma ,x\left( \sigma \right) \right) d\sigma \right)  \\ 
I_{a^{+}}^{1-\gamma ;\psi }y(a) & = & y_{0}
\end{array}
\right. 
\end{equation}
for $t\in I$, where $\overline{k}\in C\left( I^{2}\times \mathbb{R}^{n}, \mathbb{R}^{n}\right) $ for $a\leq s\leq t<\infty $, $\overline{f}\in C(I\times \mathbb{R}^{n}\times \mathbb{R},\mathbb{R}^{n})$.

\begin{teorema}
Suppose the functions $f,k$ in {\rm{Eq.(\ref{1.1})}} satisfying the conditions 
\begin{equation}
|f(t,u,v)-f(t,\overline{u},\overline{v})|\leq N(|u-\overline{u}|+|v-\overline{v}|)
\end{equation}%
and 
\begin{equation}
|\overline{k}(t,\sigma ,u)-\overline{k}(t,\sigma ,v)|\leq r(t,\sigma )|u-v|,
\end{equation}%
and 
\begin{equation}
\left\vert f\left( t,y(t),\frac{1}{\Gamma (\alpha )}\int_{a}^{t}\mathcal{W}_{\psi }^{\alpha }\left( t,\sigma ,x\left( \sigma \right) \right) \,d\sigma \right)-\overline{f}\left( t,y(t),\frac{1}{\Gamma (\alpha )}\int_{a}^{t}\overline{\mathcal{W}}_{\psi }^{\alpha }\left( t,\sigma ,x\left( \sigma \right) \right) \,d\sigma\right) \right\vert \leq \epsilon _{1}
\end{equation}
where $f,k$ and $\overline{f},\overline{k}$ are the functions involved in {\rm{Eq.(\ref{1.1})}} and {\rm{Eq.(\ref{4.1})}}, $\epsilon _{1}>0$ is an arbitrary small constant and $y(t)$ is a solution of {\rm{Eq.(\ref{4.1})}}. Then, the solution $x(t)$, $t\in I$, of {\rm{Eq.(\ref{1.1})}} depends continuously on the functions involved on the right hand side of {\rm{Eq.(\ref{1.1})}}.
\end{teorema}

\begin{proof}
Let $x(t)$ and $y(t)$ the solutions of {\rm{Eq.(\ref{1.1})}} and {\rm{Eq.(\ref{4.1})}}, respectively, and using hypotheses, we have
\begin{eqnarray}
u(t) &=&\left\vert x\left( t\right) -y\left( t\right) \right\vert   \notag
\label{eq32} \\
&=&\left\vert 
\begin{array}{c}
f\left( t,x(t),\displaystyle\frac{1}{\Gamma (\alpha )}\int_{a}^{t}\mathcal{W}%
_{\psi }^{\alpha }\left( t,\sigma ,x\left( \sigma \right) \right) \,d\sigma
\right)  \\ 
-f\left( t,y(t),\displaystyle\frac{1}{\Gamma (\alpha )}\int_{a}^{t}\mathcal{W%
}_{\psi }^{\alpha }\left( t,\sigma ,x\left( \sigma \right) \right) \,d\sigma
\right)  \\ 
+f\left( t,y(t),\displaystyle\frac{1}{\Gamma (\alpha )}\int_{a}^{t}\mathcal{W}_{\psi
}^{\alpha }\left( t,\sigma ,x\left( \sigma \right) \right) \,d\sigma \right) 
\\ 
-f\left( t,y(t),\displaystyle\frac{1}{\Gamma (\alpha )}\int_{a}^{t}\overline{%
\mathcal{W}}_{\psi }^{\alpha }\left( t,\sigma ,x\left( \sigma \right)
\right) \,d\sigma \right) 
\end{array}%
\right\vert   \notag \\
&\leq &\varepsilon _{1}+\left\vert 
\begin{array}{c}
f\left( t,x(t),\displaystyle\frac{1}{\Gamma (\alpha )}\int_{a}^{t}\mathcal{W}%
_{\psi }^{\alpha }\left( t,\sigma ,x\left( \sigma \right) \right) \,d\sigma
\right)  \\ 
-f\left( t,y(t),\displaystyle\frac{1}{\Gamma (\alpha )}\int_{a}^{t}\mathcal{W%
}_{\psi }^{\alpha }\left( t,\sigma ,x\left( \sigma \right) \right) \,d\sigma
\right) 
\end{array}%
\right\vert   \notag \\
&\leq &\varepsilon _{1}+N\left\{ |x(t)-y(t)|+\frac{1}{\Gamma (\alpha )}%
\int_{a}^{t}N_{\psi }^{\alpha }(t,\sigma )|k(t,\sigma ,x(\sigma
))-k(t,\sigma ,y(\sigma ))|\,d\sigma \right\}   \notag \\
&\leq &\varepsilon _{1}+N\left\{ |x(t)-y(t)|+\displaystyle\frac{1}{\Gamma
(\alpha )}\int_{a}^{t}N_{\psi }^{\alpha }(t,\sigma )r(t,\sigma )|x(\sigma
)-y(\sigma )|\,d\sigma \right\}   \notag \\
&=&\varepsilon _{1}+N\left\{ u(t)+\displaystyle\frac{1}{\Gamma (\alpha )}%
\int_{a}^{t}N_{\psi }^{\alpha }(t,\sigma )r(t,\sigma )u(\sigma )\,d\sigma
\right\} .
\end{eqnarray}

Note that, by Eq.(\ref{eq32}) and using the assumption $0\leq N<1$, we get
\begin{equation}
u(t)\leq \frac{\epsilon _{1}}{1-N}+\frac{N}{1-N}\left\{ \frac{1}{\Gamma (\alpha )}\int_{a}^{t}N_{\psi }^{\alpha }(t,\sigma )r(t,\sigma )u(\sigma )\,d\sigma \right\}.  \label{4.4}
\end{equation}

Now, by means of Corollary 1 we rewrite \rm{Eq.(\ref{4.4})} in the following form 
\begin{equation}
|x(t)-y(t)|\leq \left( \frac{\epsilon _{1}}{1-N}\right) \mathbb{E}_{\alpha }\left[ \frac{N}{1-N}r(t,t)(\psi (t)-\psi (\sigma ))^{\alpha }\right],\label{4.5}
\end{equation}
where $\mathbb{E}_{\alpha }(\cdot )$ is an one-parameter Mittag-Leffler function.

From \rm{Eq.(\ref{4.5})} it follows that the solution of \rm{Eq.(\ref{1.1})} depends continuously on the functions involved on the right hand side of \rm{Eq.(\ref{1.1})}.
\end{proof}

\begin{teorema}\label{ABC} Suppose the functions $f$ and $k$ in {\rm{Eq.(\ref{1.2})}} satisfying the conditions {\rm{Eq.(\ref{AB1})}} and {\rm{Eq.(\ref{BA1})}}. Furthermore, suppose that 
\begin{equation}
\left\vert \Psi ^{\gamma }\left( t,a\right) x_{0}-\Psi ^{\gamma }\left( t,a\right) y_{0}\right\vert +I_{a^{+}}^{\alpha ,\psi }\left( \left\vert 
\begin{array}{c}
f\left( s,y(s),\displaystyle\frac{1}{\Gamma (\alpha )}\int_{a}^{s}\mathcal{W}_{\psi }^{\alpha }\left( s,\sigma ,y\left( \sigma \right) \right) \,d\sigma \right) - \\ \overline{f}\left( s,y(s),\displaystyle\frac{1}{\Gamma (\alpha )}\int_{a}^{s}\overline{\mathcal{W}} _{\psi }^{\alpha }\left( s,\sigma ,y\left( \sigma \right) \right) \,d\sigma\right) 
\end{array}%
\right\vert \right) \leq \varepsilon _{2}
\end{equation}
where $f,k$ and $\overline{f},\overline{k}$ are functions involved in {\rm{Eq.(\ref{1.2})}} and {\rm{Eq.(\ref{4.2})}}, $\varepsilon _{2}>0$ is an arbitrary small constant and $y(t)$ is a solution of {\rm{
Eq.(\ref{4.2})}}. Then, the solution $x(t)$, $t\in I$ of {\rm{Eq.(\ref{1.2})}} depends continuously on the functions in right hand side of {\rm{Eq.(\ref{1.2})}}.
\end{teorema}

\begin{proof}
Let $x(t)$ and $y(t)$ the solutions of \textrm{Eq.(\ref{1.2})} and \textrm{Eq.(\ref{4.2})} and using the hypotheses we have 
\begin{eqnarray*}
u(t) &=&|x(t)-y(t)|  \notag \\
&=&\left\vert \Psi ^{\gamma }\left( t,a\right) x_{0}\right.
+I_{a^{+}}^{\alpha ,\psi }f\left( s,x(s),\frac{1}{\Gamma (\alpha )}\int_{a}^{s}W_{\psi }^{\alpha }\left( s,\sigma ,x\left( \sigma \right)\right) d\sigma \right) -  \notag \\
&&\Psi ^{\gamma }\left( t,a\right) y_{0}-\left. I_{a^{+}}^{\alpha ,\psi }\overline{f}\left( s,y(s),\frac{1}{\Gamma (\alpha )}\int_{a}^{s}W_{\psi}^{\alpha }\left( s,\sigma ,y\left( \sigma \right) \right) \,d\sigma \right)
\right\vert.
\end{eqnarray*}

As above, we add and subtract an adequate term, using $\varepsilon _{2}$ and Lemma \ref{lemmaB}, we have
\begin{eqnarray*}
u\left( t\right)  &\leq &\varepsilon _{2}+I_{a^{+}}^{\alpha ,\psi }\left\{
p(s)\left( 
\begin{array}{c}
|x(s)-y(s)| \\ 
+\displaystyle\frac{1}{\Gamma (\alpha )}\int_{a}^{s}N_{\psi }^{\alpha }(s,\sigma
)k(s,\sigma ,x(\sigma ))-k(s,\sigma ,y(\sigma ))\,d\sigma 
\end{array}%
\right) \right\}  \\
&\leq &\varepsilon _{2}+I_{a^{+}}^{\alpha ,\psi }\left\{ p(s)\left( u(s)+%
\displaystyle\frac{1}{\Gamma (\alpha )}\displaystyle\int_{a}^{s}N_{\psi }^{\alpha }(s,\sigma
)r(t,\sigma )u(\sigma )\,d\sigma \right) \right\}  \\
&=&\varepsilon _{2}+I_{a^{+}}^{\alpha ,\psi }\left\{ p(s)u(s)+\frac{p(s)}{%
\Gamma (\alpha )}\int_{a}^{s}N_{\psi }^{\alpha }(s,\sigma )r(t,\sigma
)u(\sigma )\,d\sigma \right\}  \\
&=&\varepsilon _{2}\mathbb{E}_{\alpha }\left\{ p\left( t\right) \Gamma
\left( \alpha \right) \mathbb{E}_{\alpha }\left[ r\left( t,t\right) \Gamma
\left( \alpha \right) \left( \psi \left( t\right) -\psi \left( a\right)
\right) ^{\alpha }\right] \left( \psi \left( t\right) -\psi \left( a\right)
\right) ^{\alpha }\right\} .
\end{eqnarray*}
\end{proof}

Now, we consider the following system involving a nonlinear fractional Volterra integral and a nonlinear fractional Volterra integrodifferential equations 
\begin{equation}
z(t)=h\left( t,z(t),\frac{1}{\Gamma (\alpha )}\int_{a}^{t}N_{\psi }^{\alpha }(t,\sigma )g(t,\sigma ,z(\sigma ))\,d\sigma \,,\mu \right)   \label{4.8}
\end{equation}
and 
\begin{equation}
z(t)=h\left( t,z(t),\displaystyle\frac{1}{\Gamma (\alpha )}\int_{a}^{t}N_{\psi }^{\alpha }(t,\sigma )g(t,\sigma ,z(\sigma ))\,d\sigma \,,\mu _{0}\right)   \label{4.9}
\end{equation}%
and
\begin{equation}
\left\{ 
\begin{array}{ccl}\label{eq51}
{}^{H}\mathbb{D}_{a^{+}}^{\alpha ,\beta ,\psi }z(t) & = & h\left( t,z(t),\displaystyle\frac{1}{\Gamma (\alpha )}\int_{a}^{t}N_{\psi }^{\alpha }(t,\sigma)g(t,\sigma ,z(\sigma ))\,d\sigma \,,\mu \right)  \\ 
I_{a^{+}}^{1-\gamma ,\psi }z(a) & = & z_{0}
\end{array}%
\right. 
\end{equation}
and 
\begin{equation}\label{eq52}
\left\{ 
\begin{array}{ccl}
{}^{H}\mathbb{D}_{a^{+}}^{\alpha ,\beta ,\psi }z(t) & = & h\left( t,z(t), \displaystyle\frac{1}{\Gamma (\alpha )}\int_{a}^{t}N_{\psi }^{\alpha }(t,\sigma )g(t,\sigma ,z(\sigma ))\,d\sigma \,,\mu _{0}\right)  \\ 
I_{a^{+}}^{1-\gamma ,\psi }z(a) & = & z_{0}
\end{array}
\right. 
\end{equation}%
for $t\in I$, where $g\in C(I^{2}\times \mathbb{R}^{n},\mathbb{R}^{n})$, $a\leq \sigma \leq t<\infty $ and $h\in C(I\times \mathbb{R}^{n}\times \mathbb{R}^{n}\times \mathbb{R},\mathbb{R}^{n})$.

\begin{teorema}\label{teo10}
Suppose the functions $h,g$ in {\rm{Eq.(\ref{4.8})}} and {\rm{Eq.(\ref{4.9})}} satisfying the conditions 
\begin{equation} \label{4.12}
|h(t,u,v,\mu )-h(t,\overline{u},\overline{v},\mu )|\leq \overline{N}(|u-\overline{u}|+|v-\overline{v}|) 
\end{equation}
\begin{equation}\label{4.13}
|h(t,u,v,\mu )-h(t,u,v,\mu _{0})|\leq q(t)|\mu -\mu _{0}|  
\end{equation}
\begin{equation}\label{4.14}
|g(t,\sigma ,u)-g(t,\sigma ,v)|\leq \overline{r}(t,\sigma )|u-v|
\end{equation} 
where $0\leq \overline{N}<1$ is a constant, $q\in C(I,\mathbb{R}_{+})$ such that $q(t)\leq Q<\infty $, $Q$ is a constant and $\overline{r}(t,\sigma )\in C(D,\mathbb{R}_{+})$ in which $D$ is defined as in {\rm Lemma \ref{lemmA}.} Let $z_{1}(t)$ and $z_{2}(t)$ be the solutions of {\rm{Eq.(\ref{4.8})}} and {\rm{Eq.(\ref{4.9})}}, respectively.

Then, 
\begin{equation}
|z_{1}(t)-z_{2}(t)|\leq Q\frac{|\mu -\mu _{0}|}{1-\overline{N}}\mathbb{E}_{\alpha }\left[ \frac{\overline{N}}{1-\overline{N}}r(t,t)(\psi (t)-\psi (a))^{\alpha }\right] 
\end{equation}
where $\mathbb{E}_{\alpha }(\cdot )$ is an one-parameter Mittag-Leffler function.
\end{teorema}

\begin{proof}
Let $x(t)$ and $y(t)$ the solutions of {\rm{Eq.(\ref{4.8})}} and {\rm{Eq.(\ref{4.9})}}, for $t\in I$, and the hypotheses, we have 
\begin{eqnarray*}
z(t) &=&|z_{1}(t)-z_{2}(t)| \\
&=&\left\vert h\left( t,z_{1}(t),\frac{1}{\Gamma (\alpha )}\int_{a}^{t}N_{\psi }^{\alpha }(t,\sigma )g(t,\sigma ,z_{1}(\sigma ))\,d\sigma \,,\mu \right) \right. - \\
&&h\left( t,z_{2}(t),\frac{1}{\Gamma (\alpha )}\int_{a}^{t}N_{\psi }^{\alpha }(t,\sigma )g(t,\sigma ,z_{2}(\sigma ))\,d\sigma \,,\mu \right) + \\
&&h\left( t,z_{2}(t),\frac{1}{\Gamma (\alpha )}\int_{a}^{t}N_{\psi }^{\alpha }(t,\sigma )g(t,\sigma ,z_{2}(\sigma ))\,d\sigma \,,\mu \right) + \\
&&\left. h\left( t,z_{2}(t),\frac{1}{\Gamma (\alpha )}\int_{a}^{t}N_{\psi }^{\alpha }(t,\sigma )g(t,\sigma ,z_{2}(\sigma ))\,d\sigma \,,\mu _{0}\right) \right\vert .
\end{eqnarray*}

Proceeding as in Theorem \ref{ABC}, can write 
\begin{eqnarray}
z(t) &\leq &\overline{N}\left\{ |z_{1}(t)-z_{2}(t)|+\frac{1}{\Gamma (\alpha )%
}\int_{a}^{t}N_{\psi }^{\alpha }|g(t,\sigma ,z_{1}(\sigma ))-g(t,\sigma
,z_{2}(\sigma ))|\,d\sigma \right\}   \notag  \label{4.16} \\
&&+q(t)|\mu -\mu _{0}|  \notag \\
&=&\overline{N}\left\{ z(t)+\frac{1}{\Gamma (\alpha )}\int_{a}^{t}N_{\psi
}^{\alpha }\overline{r}(t,\sigma )z(\sigma )\,d\sigma \right\} +Q|\mu -\mu
_{0}|
\end{eqnarray}

As $0\leq \overline{N}<1$, {\rm{Eq.(\ref{4.16})}} can be rewritten as follows 
\begin{equation}
z(t)\leq \frac{Q|\mu -\mu _{0}|}{1-\overline{N}}+\frac{\overline{N}}{1-\overline{N}}\frac{1}{\Gamma (\alpha )}\int_{a}^{t}N_{\psi }^{\alpha }(t,\sigma )\overline{r}(t,\sigma )z(\sigma )\,d\sigma .  \label{4.17}
\end{equation}

Using Corollary \ref{coro1} in {\rm{Eq.(\ref{4.17})}}, we conclude that 
\begin{equation*}
|z_{1}(t)-z_{2}(t)|\leq \frac{Q|\mu -\mu _{0}|}{1-\overline{N}}\mathbb{E}_{\alpha }\left[ \frac{\overline{N}}{1-\overline{N}}r(t,t)(\psi (t)-\psi(a))^{\alpha }\right] 
\end{equation*}
where $\mathbb{E}_{\alpha }(\cdot )$ is an one-parameter Mittag-Leffler function. 
\end{proof}

\begin{teorema}
Suppose the functions $h,g$ in {\rm{Eq.(\ref{eq51})}} and {\rm{Eq.(\ref{eq52})}}, satisfying the conditions {\rm{Eq.(\ref{4.12})}}-{\rm{Eq.(\ref{4.14})}} with $\overline{p}(t)$ in the place of $\overline{N}$ in {\rm{Eq.(\ref{4.12})}}, where $\overline{p}\in C(I,\mathbb{R}_{+})$ and the function $q(t)$ in {\rm{Eq.(\ref{4.13})}} be such that 
\begin{equation*}
\frac{1}{\Gamma(\alpha)}\int_{a}^{t}N^{\alpha}_{\psi}(t,s) q(s)\,ds\leq \overline{Q}<\infty 
\end{equation*}
where $\overline{Q}$ is a constant. Let $z_{1}(t)$ and $z_{2}(t)$ be the solutions of {\rm{Eq.(\ref{eq51})}} and {\rm{Eq.(\ref{eq52})}} Then, 
\begin{equation*}
|z_{1}(t)-z_{2}(t)|\leq Q|\mu -\mu _{0}|\mathbb{E}_{\alpha }\left\{ \overline{p}\left( t\right) \Gamma \left( \alpha \right) \mathbb{E}_{\alpha }\left( \overline{r}\left( t,t\right) \Gamma \left( \alpha \right) \left( \psi \left( t\right)-\psi \left( a\right) \right) ^{\alpha }\right) \left( \psi \left( t\right) -\psi \left( a\right) \right) ^{\alpha }\right\}.
\end{equation*}
\end{teorema}

\begin{proof}
As before in Theorem \ref{teo10}.
\end{proof}

\section{Concluding remarks}

We conclude this article with the objectives achieved, that is, we carry out a brief study on existence, uniqueness, solution estimate, and 
continuous dependence of solutions of the nonlinear fractional Volterra integral equation, Eq.(\ref{1.1}), and the nonlinear fractional integrodifferential equation, Eq.(\ref{1.2}). For this end, we introduce the metric, Eq.(\ref{AA1}) and the norm, Eq.(\ref{AA2}), as well as Lemma \ref{lemmA}, Lemma \ref{lemmaB} and Corollary \ref{coro1}, which are fundamental to obtain our main results. In this sense, we contribute to the growth of the fractional calculus, particularly in the case of fractional differential equations and fractional integral equations, especially involving a recent and general formulation of the fractional derivative, the so-called $\Psi$-Hilfer fractional derivative. However, as seen in the text, there are numerous types of differential equations, integral equations and consequently problems that should be investigated. We emphasize that one of the problems that deserves special mention comes from the impulsive equations, which will be object of studies whose results will be published in a future work.
\bibliography{ref}
\bibliographystyle{plain}

\end{document}